\documentclass[11pt]{amsart}
\usepackage[margin=1.1in]{geometry}
\usepackage{amsmath,amssymb,amsthm}
\usepackage{tikz-cd}
\usepackage[colorlinks=true,linkcolor=blue,citecolor=blue]{hyperref}

\theoremstyle{plain}
\newtheorem{theorem}{Theorem}[section]
\newtheorem{lemma}[theorem]{Lemma}
\newtheorem{proposition}[theorem]{Proposition}
\newtheorem{corollary}[theorem]{Corollary}
\theoremstyle{definition}
\newtheorem{definition}[theorem]{Definition}
\newtheorem{notation}[theorem]{Notation}
\newtheorem{example}[theorem]{Example}
\newtheorem{remark}[theorem]{Remark}
\newtheorem*{theoremA}{Theorem A}
\newtheorem*{theoremB}{Theorem B}

\newcommand{\ann}{\operatorname{ann}}
\newcommand{\Soc}{\operatorname{Soc}}

\newcommand{\Z}{\mathbb{Z}}
\newcommand{\GE}{\Gamma_E}
\newcommand{\GA}{\Gamma_A}
\newcommand{\dg}{\operatorname{diam}}
\newcommand{\m}{\mathfrak{m}}

\begin{document}

\title[Zero-divisors of Gorenstein rings]{Zero-divisors of Gorenstein rings}

\author{Ganesh S. Kadu}
\address{Ganesh S. Kadu, Department of Mathematics, Savitribai Phule Pune University, Pune, India}
\email{ganeshkadu@gmail.com}

\author{Vishnu B. Tanpure}
\address{Vishnu B. Tanpure, Department of Mathematics, Savitribai Phule Pune University, Pune, India}
\email{vbtanpure@gmail.com}


\begin{abstract}
Let $R$ be a commutative Artinian ring. We consider two graphs associated to $R$, namely the
compressed zero-divisor graph $\GE(R)$ and the associate class graph $\GA(R)$. Partitioning
the vertex set of a zero-divisor graph into its \emph{core} and its \emph{boundary}, we count
the core vertices that dominate the core. This count is a graph invariant, and we estimate it
for  $\Gamma(R)$, $\GA(R)$ and $\GE(R)$. We prove that the count for $\GA(R)$ is
bounded below by the count for $\GE(R)$, and that the lower bound is attained  precisely when 
$R$ is Gorenstein. As a consequence we obtain that $R$ is Gorenstein if and only if
$\GA(R)\cong\GE(R)$ as graphs, the isomorphism being an arbitrary one and not merely the
natural compression map. Using the same counting technique we then answer, for Artinian rings, a question of Anderson
and LaGrange by showing that $\Gamma(R)\cong\GE(R)$ if and only if
$R\cong\Z_2^{\,n}$ for some $n\ge2$, or $R\cong\Z_4$, or $R\cong\Z_2[x]/(x^2)$.
\end{abstract}

\maketitle

\section{Introduction}

Let $R$ be a commutative ring with unity. In \cite{Beck}, Beck associated to a ring $R$ a
simple graph with vertex set consisting of all elements of $R$, two vertices $a,b\in R$ being
adjacent whenever $ab=0$. This graph is known as the zero-divisor graph associated to a ring,
and since its introduction by Beck the zero-divisor graph has been a fruitful theme of
research. The zero-divisor graph was later modified by D. F. Anderson and P. S. Livingston in
\cite{AL} to include only the non-zero zero-divisors of $R$ as its vertex set, two vertices
being adjacent if their product is zero. This graph is denoted by $\Gamma(R)$. The study of
these graphs offers a rich interplay between the ring theoretic properties of $R$ and the
graph theoretic properties of $\Gamma(R)$, see \cite{ALB,LaG16,ALS,LaG07,LaG08,Red}.

A compressed version of $\Gamma(R)$ was introduced and studied by S. Mulay in \cite{Mulay} and
is denoted by $\GE(R)$. This graph, known as the compressed zero-divisor graph of $R$, is
obtained from $\Gamma(R)$ by letting two vertices $a,b$ be equivalent $(a\approx b)$ if they
have the same annihilators, i.e., $\ann(a)=\ann(b)$. Thus the vertex set of $\GE(R)$ is the
set of equivalence classes, termed annihilator classes, of vertices of $\Gamma(R)$, and the
adjacency is inherited from that in $\Gamma(R)$, i.e., a class $[a]_E$ is adjacent to $[b]_E$
if $ab=0$. We refer the reader to the article by J. Coykendall et al. \cite{CSSS} dealing with
various aspects of $\GE(R)$, and also to \cite{ALB,LaG16,Ax,Kadu,SW}. Tang, Lin and Wu
introduced a related notion in \cite{TLW}, where the non-zero zero-divisors $a, b$ in $R$ are equivalent $(a \sim b )$ if $a$ and $b$  are associates ($a=ub$ for some unit $u$ in $R$). The equivalence classes, termed as the associate classes, are the vertices of the graph $\Gamma_A(R)$ and two vertices $[a]_A$ and $[b]_A$ are adjacent if $ab=0$.  The compression with
respect to associate classes is finer than that with respect to annihilator classes, and hence
the graph $\GA(R)$, called the associate class graph, tends to reflect the algebra of $R$ more
accurately. For categorical properties of the functor $R\rightsquigarrow\GA(R)$ on finite
rings see A. Duric et al. \cite{DJS}.

For $a,b\in R$ the
relation of being associates (denoted $a\sim b$) is finer than the relation of having equal
annihilators, that is, $a\sim b$ implies $a\approx b$. Thus, every annihilator class of $R$ is a
disjoint union of associate classes, and $\GE(R)$ may be regarded as a compression of both
$\GA(R)$ and $\Gamma(R)$. The three graphs are compared by the natural maps $\eta,\psi,\varphi$
which form a commutative diagram $\eta=\varphi\circ\psi$,

\[
\begin{tikzcd}[column sep=large,row sep=large]
	\Gamma(R) \arrow[r,"\psi"] \arrow[dr,swap,"\eta"] & \GA(R)  \arrow[d,swap,"\varphi"]\\
	& \GE(R)
\end{tikzcd}
\]
where $\eta(a)=[a]_E$, $\psi(a)=[a]_A$ and $\varphi([a]_A)=[a]_E$.

Our first objective is the comparison of $\GE(R)$ and $\GA(R)$. In Theorem \ref{thm:gor-nat}
we show, for an Artinian ring $R$, that $\approx$ and $\sim$ coincide if and only if $R$ is
Gorenstein, and hence that natural map $\varphi$ is an isomorphism if and only if $R$ is Gorenstein. This
leaves open question of whether
the existence of some isomorphism $\GA(R)\cong\GE(R)$ also forces $R$ to be
Gorenstein. We answer this affirmatively. The method is a counting argument which we now
describe, and which is the technical heart of the paper.

For an Artinian ring one has $\dg\Gamma(R)\le3$, and we partition the vertex set of a graph
$G$ of diameter at most $3$ into its \emph{boundary}, consisting of those vertices which
realise the distance $3$, and its \emph{core}, consisting of the remaining ones. Since a graph
isomorphism preserves distances, this partition is preserved, and so is the set of \emph{core-dominating} vertices,
\[
   D(G)=\{\,v\in V(G) \mid \ v \text{ is adjacent to every vertex of } C(G)\setminus\{v\}\,\}.
\]
 In Theorem \ref{thm:core} we determine $C(\Gamma(R))$
completely for an arbitrary Artinian ring, and in Section \ref{sec:count} we compute
$|D(\Gamma(R))|$ and $|D(\GE(R))|$ and obtain a lower bound for $|D(\GA(R))|$. Writing
$R\cong R_1\times\cdots\times R_k\times K_1\times\cdots\times K_m$ with the $R_i$ non-reduced
Artinian local and the $K_j$ fields, and letting $P_i$ denote the number of associate classes
contained in $\Soc R_i\setminus\{0\}$, we obtain
\[
   |D(\GE(R))|=2^{\,k}-1+m+\delta,\qquad
   |D(\GA(R))|\;\ge\;\prod_{i=1}^{k}\bigl(1+P_i\bigr)-1+m ,
\]
with $\delta\in\{0,1\}$ as in Proposition \ref{prop:DE}. The point is that $P_i=1$ exactly
when $R_i$ is Gorenstein, while $P_i\ge3$ otherwise.  The lower bound is therefore attained
only in the Gorenstein case. This yields our first main theorem (Theorem \ref{thm:mainA}).

\begin{theoremA}
Let $R$ be an Artinian ring. Then $R$ is Gorenstein if and only if $\GA(R)\cong\GE(R)$ as
graphs.
\end{theoremA}

The notion of Gorenstein rings is well studied in commutative algebra (see \cite{Bass} by H. Bass).
For an Artinian local ring $(R,\mathfrak m,k)$ the Gorenstein condition simply says that
$\Soc R$ is a principal ideal, equivalently that $\dim_k\Soc R=1$. Theorem A thus gives a
purely graph theoretic characterization of the Gorenstein property among Artinian rings.

Our second objective is a question of Anderson and LaGrange \cite[Question 2.12]{ALB}, (cf. \cite[Question 4.10]{LaG16}), asking for the rings $R$ with
$\Gamma(R)\cong\GE(R)$. The same counting technique settles this question for Artinian rings. Indeed
$|D(\GE(R))|$ is always finite whereas $|D(\Gamma(R))|=\prod_i|\Soc R_i|-1+t,$ where $t$ is the number of copies of $\mathbb Z_2$ appearing in the decomposition of $R.$ An isomorphism between $\Gamma(R)$ and $\Gamma_E(R)$
therefore forces every socle to be finite, hence every residue field to be finite, hence $R$ itself to be finite. We then invoke \cite[Corollary 2.10]{ALB} to obtain our second main theorem (Theorem \ref{thm:mainB}).


\begin{theoremB}
Let $R$ be an Artinian ring which is not a field. Then $\Gamma(R)\cong\GE(R)$ if and only if
$R\cong\Z_2^{\,n}$ for some $n\ge2$, or $R\cong\Z_4$, or $R\cong\Z_2[x]/(x^2)$.
\end{theoremB}

\section{Preliminaries}

Throughout, $R$ is a commutative ring with unity. We recall the notions we need.

\smallskip
\noindent\textbf{2.1.} $Z(R)$ denotes the set of zero-divisors of $R$, with $0\in Z(R)$, and
$Z^{*}(R)=Z(R)\setminus\{0\}$. Denote by $N(R)$  the nilradical of $R.$ We write $R^{\times}$ for the group of units. The zero-divisor
graph $\Gamma(R)$ is the simple graph with vertex set $Z^{*}(R)$, two vertices $a,b$ being
adjacent if $ab=0$; it was introduced by Anderson and Livingston \cite{AL}.

\smallskip
\noindent\textbf{2.2.} For $a,b\in R$ we write $a\approx b$ if $\ann(a)=\ann(b)$. This is an
equivalence relation on $R$; its classes are the annihilator classes, and the class of
$a\in V(\Gamma(R))$ is denoted $[a]_E=\{x\in R\mid\ann(x)=\ann(a)\}$. The compressed
zero-divisor graph $\GE(R)$, introduced by Mulay \cite{Mulay}, has vertex set
$V(\Gamma(R))/\!\approx$, two distinct vertices $[a]_E,[b]_E$ being adjacent if $ab=0$.

\smallskip
\noindent\textbf{2.3.} We write $a\sim b$ if $a=ub$ for some $u\in R^{\times}$; the classes of
this equivalence relation are the associate classes, denoted $[a]_A=\{ua\mid u\in R^{\times}\}$.
The associate class graph $\GA(R)$, introduced by Tang, Lin and Wu \cite{TLW}, has vertex set
$V(\Gamma(R))/\!\sim$, two distinct vertices being adjacent if their product is $0$.

\smallskip
\noindent\textbf{2.4.} The associate classes are exactly the orbits of the natural action of
$G=R^{\times}$ on $X=R$ by multiplication. The relation $\sim$ is finer than $\approx$, so every annihilator class is a
disjoint union of associate classes.

\smallskip
\noindent\textbf{2.5.} Every annihilator ideal is a union of annihilator classes: if
$x\in\ann(a)$ and $\ann(y)=\ann(x)$, then $a\in\ann(x)=\ann(y)$ and so $y\in\ann(a)$. The same
holds for associate classes, since $\sim$ is finer than $\approx$.

\smallskip
\noindent\textbf{2.6.} A local ring $R$ with maximal ideal $\mathfrak m$ and residue field
$k=R/\mathfrak m$ is denoted $(R,\mathfrak m,k)$. For such a ring, $R^{\times}=R\setminus
\mathfrak m$; if $R$ is in addition Artinian, then $Z(R)=\mathfrak m$.

\smallskip
\noindent\textbf{2.7.} For a local ring $(R,\mathfrak m,k)$ the ideal
$\Soc R=(0:_R\mathfrak m)=\{x\in R\mid x\mathfrak m=0\}$ is called the socle of $R$; it is a
vector space over $k$. An Artinian local ring $(R,\mathfrak m,k)$ is Gorenstein if
$\dim_k\Soc R=1$. An Artinian ring is Gorenstein if each of its local factors is Gorenstein. In
particular, every field is Gorenstein, and a finite product is Gorenstein if and only if each
factor is.

\smallskip
\noindent\textbf{2.8.} For an Artinian local ring $(R,\mathfrak m,k)$ the following are
equivalent, see \cite[3.2.15]{BH}: (i) $R$ is Gorenstein; (ii) all finite $R$-modules are
reflexive; (iii) $I=\ann(\ann I)$ for every ideal $I$ of $R$; (iv) $I\cap J\neq0$ for all
non-zero ideals $I,J$.

\smallskip
\noindent\textbf{2.9.} For  a non-reduced Artinian local ring $(R,\mathfrak m,k)$ we denote by
$P(R)$  the number of associate classes contained in $\Soc R\setminus\{0\}.$

\smallskip
\noindent\textbf{2.10.} Every Artinian ring is, uniquely up to order, a finite direct product
of Artinian local rings. Annihilators in a product are computed coordinatewise:
$\ann(x)=\prod_t\ann_{T_t}(x_t)$ for $x=(x_t)\in\prod_tT_t$. The same formula holds for
associate classes, since $\bigl(\prod_tT_t\bigr)^{\times}=\prod_tT_t^{\times}$.

\smallskip
\noindent\textbf{2.11.} By \cite[Theorem 2.3]{AL}, $\Gamma(R)$ is connected with
$\dg\Gamma(R)\le3$. A path in $\Gamma(R)$ maps to a walk under $\eta$ and under $\psi$, so
$\GA(R)$ and $\GE(R)$ are connected with $\dg\GA(R)\le3$ and $\dg\GE(R)\le3$ as well.

\smallskip
\noindent\textbf{2.12.} A vertex of a graph is called universal if it is adjacent to every other vertex.

\section{Gorenstein rings and the two equivalence relations}\label{sec:gor}

In this section we compare the graphs $\GA(R)$ and $\GE(R)$,  and show that for Artinian rings $\GA(R) \cong \GE(R)$ (by a natural map) if and only if $R$ is Gorenstein ring. Later, in Section 5 (see Theorem \ref{thm:mainA}) we extend this result to show that  $\GA(R) \cong \GE(R)$ (by an arbitrary graph isomorphism) if and only if $R$ is Gorenstein ring.

\begin{lemma}\label{lem:socorbits}
Let $(R,\mathfrak m,k)$ be a non-reduced Artinian local ring and let
$P(R)$ denote the number of associate classes contained in $\Soc R\setminus\{0\}$. Then,
\begin{enumerate}
\item[(1)] $[s]_A=k^{\times}s$ for every $0\neq s\in\Soc R$, that is, the associate classes
inside $\Soc R\setminus\{0\}$ are exactly the punctured lines of the $k$-vector space
$\Soc R$;
\item[(2)] $P(R)=1$ if and only if $R$ is Gorenstein, and $P(R)\ge3$ otherwise.
\end{enumerate}
\end{lemma}

\begin{proof}
(1) Let $u\in R^{\times}$ and $s\in\Soc R$. If $u'\in R^{\times}$ has the same residue as $u$,
then $u-u'\in\mathfrak m$ and hence $(u-u')s=0.$ So, element $us$ depends only on $\bar u\in
k^{\times}$. Conversely every element of $k^{\times}$ lifts to a unit of $R$. Hence
$[s]_A=\{us\mid u\in R^{\times}\}=k^{\times}s$, which is the set of non-zero elements of the
line $ks$. Distinct lines meet only in $0$, so the classes inside $\Soc R\setminus\{0\}$
correspond bijectively to the lines of $\Soc R$.

(2) If $\dim_k\Soc R=1$ there is exactly one line, so we obtain $P(R)=1$. If $\dim_k\Soc R\ge2$, choose
$k$-linearly independent $s,t\in\Soc R.$ Then $ks$, $kt$ and $k(s+t)$ are three distinct
lines, so $P(R)\ge3$.
\end{proof}

The next example shows that $\GE(R)$ and $\GA(R)$ need not be isomorphic.

\begin{example}\label{ex:basic}
Let $R=\Z_2[x,y]/(x^2,xy,y^2)$. Here $\mathfrak m^2=0$ and $\Soc R=\mathfrak m$ has dimension
$2$, so $R$ is not Gorenstein. One checks at once that $\GE(R)$ is a graph with a single
vertex, while $\GA(R)$ is a complete graph on the three vertices $[\bar x]_A$, $[\bar y]_A$
and $[\overline{x+y}]_A$. Note that $P(R)=3$, in accordance with Lemma \ref{lem:socorbits}.
\end{example}

\begin{theorem}\label{thm:gor-nat}
For an Artinian ring $R$ the following are equivalent.
\begin{enumerate}
\item[(1)] $R$ is Gorenstein;
\item[(2)] $\approx$ and $\sim$ define the same equivalence relation on $R$;
\item[(3)] the natural map $\varphi:V(\GA(R))\to V(\GE(R))$, $\varphi([a]_A)=[a]_E$, is a
graph isomorphism.
\end{enumerate}
\end{theorem}

\begin{proof}
We first prove the equivalence of (1) and (2) in the local case, so let $(R,\mathfrak m,k)$
be Artinian local. Assume $R$ is Gorenstein and let $a\approx b$, i.e., $\ann(a)=\ann(b)$.
Then $\ann(\ann(a))=\ann(\ann(b))$, so $(a)=(b)$ by 2.8(iii). As $R$ is local this gives
$a=ub$ for some unit $u$, i.e., $a\sim b$. The converse implication $a\sim b\Rightarrow
a\approx b$ is clear, so $\approx$ and $\sim$ agree.

Conversely, assume $\approx$ and $\sim$ agree. Choose $n$ with $\mathfrak m^{n-1}\neq0$ and
$\mathfrak m^{n}=0$, and let $0\neq a\in\mathfrak m^{n-1}$. Then $a\mathfrak m=0$, so
$\ann(a)=\mathfrak m$ and $a\in\Soc R$. For any $x\in\Soc R\setminus\{0\}$ we have
$\ann(x)=\mathfrak m=\ann(a).$ Hence, $x\approx a$ and therefore $x\sim a$, i.e., $x=va$ for
some unit $v$. Thus $\Soc R=(a)$ and $\dim_k\Soc R=1$, so $R$ is Gorenstein. 

For a general Artinian ring write $R\cong T_1\times\cdots\times T_n$ with $T_t$ Artinian
local, as in 2.10. By 2.10, for $a,b\in R$ we have $\ann(a)=\ann(b)$ if and only if
$\ann(a_t)=\ann(b_t)$ for all $t$, and likewise $a\sim b$ if and only if $a_t\sim b_t$ for all
$t$. Hence $\approx$ and $\sim$ agree on $R$ if and only if they agree on every factor. Since
$R$ is Gorenstein if and only if each $T_t$ is Gorenstein, the equivalence of (1) and (2) follows from
the local case.

Finally, $\varphi$ is always surjective, and it is injective precisely when $\approx$ and
$\sim$ agree on $Z^{*}(R)$. In this case $[a]_A=[a]_E$ for all $a$, so $\varphi$ preserves  adjacency and is an isomorphism. Conversely, if $\varphi$ is an isomorphism then it
is injective, so $\approx$ and $\sim$ agree on $Z^{*}(R).$ As $\approx$ and $\sim$ trivially agree on units
and on $0$ and, since $R$ is Artinian every element is either a unit or zero-divisor, they agree on $R$. This gives (2) $\Leftrightarrow$ (3).
\end{proof}


\begin{remark}
In section 5, we extend the result above and show that (1), (2) and (3) are equivalent  even when the natural map $\varphi$ in (3) is replaced by an arbitrary graph isomorphism, that is,  $\GA(R) \cong \GE(R)$ (by an arbitrary graph isomorphism) if and only if $R$ is a Gorenstein ring. 	
\end{remark}
We shall also need the following two facts. The first says that, apart from a degenerate case,
the compressed zero-divisor graph of an Artinian local ring has only one universal vertex. Recall from 2.12 that a
vertex of a graph is universal if it is adjacent to every other vertex of the graph.

\begin{lemma}\label{lem:unique-univ}
Let $(R,\mathfrak m)$ be an Artinian local ring and let $[s]_E=\Soc R\setminus\{0\}$, which is
a universal vertex of $\GE(R)$. If $[x]_E\neq[s]_E$ is also a universal vertex of $\GE(R)$,
then $|V(\GE(R))|=2$.
\end{lemma}

\begin{proof}
Since $\Soc R=\ann(\mathfrak m)$ and $Z(R)=\mathfrak m$, every non-zero element of $\Soc R$
annihilates every zero-divisor, so $[s]_E$ is indeed universal. Let $[x]_E\neq[s]_E$ be
universal. Then $\mathfrak m\setminus[x]_E\subseteq\ann(x)$. If $x^2=0$ then
$[x]_E\subseteq\ann(x)$ by 2.5, so $\ann(x)=\mathfrak m$ and $[x]_E=[s]_E$, a contradiction;
hence $x^2\neq0$.

Let $[y]_E$ be any vertex different from $[x]_E$. Note that $x+y \neq 0,$ otherwise $[x]_E=[y]_E$. Then $xy=0$, and
$x(x+y)=x^2+xy=x^2\neq0$, so $[x]_E$ is not adjacent to $[x+y]_E$.  Now, universality of $[x]_E$
forces $[x+y]_E=[x]_E$, i.e., $x+y\in[x]_E$. Since $[x]_E$ is adjacent to $[y]_E$ and
$x+y\in[x]_E$, we get $y(x+y)=0$, i.e., $y^2=0$. We claim $\ann(y)=\mathfrak m$. Let
$z\in\mathfrak m$. If $z\in[x]_E$ then $\ann(z)=\ann(x)\ni y$, so $yz=0$. If $z\notin[x]_E$
then $[z]_E\neq[x]_E$, so $xz=0$ and $(x+y)z=0$ by universality of $[x]_E$ and the fact that
$x,x+y\in[x]_E$; subtracting gives $yz=0$. Hence $\mathfrak m\subseteq\ann(y)$, and as $y$ is
a non-zero zero-divisor, $\ann(y)=\mathfrak m$ and $[y]_E=[s]_E$. Therefore
$V(\GE(R))=\{[x]_E,[s]_E\}$.
\end{proof}



\section{The core of the zero-divisor graph}\label{sec:core}

By 2.11 all three graphs $\Gamma(R)$, $\Gamma_A(R)$ and $\Gamma_E(R)$ have diameter at most $3$. This suggests the following partition of
their vertex sets. 

\begin{definition}\label{def:core}
Let $G$ be a connected graph with $\dg G\le3$. The \emph{boundary} of $G$ is
$B(G)=\{x\in V(G) \mid \ d(x,y)=3\text{ for some }y\in V(G)\}$, and the \emph{core} of $G$ is
$C(G)=V(G)\setminus B(G)=\{x\in V(G) \mid \ d(x,y)\le2\text{ for all }y\in V(G)\}$. Thus
$V(G)=C(G)\,\dot\cup\,B(G)$, and any graph isomorphism $G\cong H$ carries $C(G)$ onto $C(H)$
and $B(G)$ onto $B(H)$, since isomorphisms preserve distances.
\end{definition}

The next theorem shows that the three graphs have the same core and boundary, in the sense
that these can be determined from the representatives.

\begin{theorem}\label{thm:dist3}
Let $R$ be a commutative ring and let $x,y\in Z^{*}(R)$. Then the following are equivalent:
\textup{(i)} $d_{\Gamma(R)}(x,y)=3$; \textup{(ii)} $d_{\GA(R)}([x]_A,[y]_A)=3$;
\textup{(iii)} $d_{\GE(R)}([x]_E,[y]_E)=3$.
\end{theorem}

\begin{proof}
(i)$\Rightarrow$(ii). Assume $d_{\Gamma(R)}(x,y)=3$ and suppose
$d_{\GA(R)}([x]_A,[y]_A)\le2$. If $[x]_A=[y]_A$ then $x$ and $y$ have the same neighbours,
forcing $d(x,y)\le2$. If $[x]_A$ and $[y]_A$ are adjacent then $xy=0$ and $d(x,y)=1$. If
$[x]_A-[z]_A-[y]_A$ is a path then $xz=zy=0$, so $x-z-y$ is a path in $\Gamma(R)$ and
$d(x,y)\le2$. Each case contradicts $d(x,y)=3$. As $\dg\GA(R)\le3$ by 2.11, we get
$d_{\GA(R)}([x]_A,[y]_A)=3$.

(ii)$\Rightarrow$(i). Assume $d_{\GA(R)}([x]_A,[y]_A)=3$. In particular, $[x]_A\neq[y]_A$, so
$x\neq y$. If $xy=0$ then $[x]_A$ and $[y]_A$ are adjacent, a contradiction. If $x-z-y$ is a
path in $\Gamma(R)$, then $xz=zy=0$ and hence $d_{\GA(R)}([x]_A,[y]_A)\le2$, again a
contradiction. So $d_{\Gamma(R)}(x,y)=3$.

The proofs of (i)$\Leftrightarrow$(iii) are identical, with $[\,\cdot\,]_A$ replaced by
$[\,\cdot\,]_E$.
\end{proof}

\begin{corollary}\label{cor:core-transfer}
Let $R$ be an Artinian ring and $x\in Z^{*}(R)$. Then $x\in B(\Gamma(R))$ if and only if
$[x]_A\in B(\GA(R))$, if and only if $[x]_E\in B(\GE(R))$. Consequently the same holds for core, that is, with
$B$ replaced by $C$, and
\[
   C(\GE(R))=\{[x]_E:x\in C(\Gamma(R))\},\qquad C(\GA(R))=\{[x]_A:x\in C(\Gamma(R))\} .
\]
\end{corollary}

We now determine the core of $\Gamma(R)$ for an arbitrary Artinian ring.

\begin{notation}\label{not:prod}
From now on $R$ is an Artinian ring which is not a field. By the structure theorem for Artinian rings we write $R$ as
\[
   R\;\cong\;R_1\times\cdots\times R_k\times K_1\times\cdots\times K_m ,
\]
where each $(R_i,\mathfrak m_i,k_i)$ is a non-reduced Artinian local ring, so
$\mathfrak m_i\neq0$ and $\Soc R_i\neq0$, and each $K_j$ is a field. Put $n=k+m$ and write
$R=T_1\times\cdots\times T_n$ with $T_i=R_i$ for $i\le k$ and $T_{k+j}=K_j$. For
$x=(x_1,\dots,x_n)\in R$ set
\[
   \mathcal N(x)=\{t \mid \ x_t\text{ is a non-unit of }T_t\},\qquad
   \mathcal U(x)=\{1,\dots,n\}\setminus\mathcal N(x).
\]
Thus $x\in Z(R)$ if and only if $\mathcal N(x)\neq\emptyset$, and
$N(R)=\mathfrak m_1\times\cdots\times\mathfrak m_k\times0\times\cdots\times0$ is the set of
$x$ with $\mathcal N(x)=\{1,\dots,n\}$. We write $e_t$ for the idempotent whose $t^{\text{th}}$
coordinate is $1$ and whose other coordinates are $0$, and $\varepsilon_j:=e_{k+j}$, so that
$\varepsilon_jR$ is a minimal ideal of $R.$ Further, $ae_t$ denotes the element with $t^{\text{th}}$
coordinate $a$ and all other coordinates $0$.
\end{notation}

\begin{lemma}\label{lem:dist}
Let $x,y\in Z^{*}(R)$ with $x\neq y$. Then $d_{\Gamma(R)}(x,y)=3$ if and only if
$\mathcal N(x)\cap\mathcal N(y)=\emptyset$ and $xy\neq0$.
\end{lemma}

\begin{proof}
We first claim that $d(x,y)\le2$ if and only if $xy=0$ or $\ann(x)\cap\ann(y)\neq0$. Indeed,
if $xy=0$ then $d(x,y)=1$; and if $0\neq z\in\ann(x)\cap\ann(y)$ then either $z\notin\{x,y\}$,
in which case $x-z-y$ is a path, or $z\in\{x,y\}$, in which case $xy=0$. Conversely a common
neighbour of $x$ and $y$ is a non-zero element of $\ann(x)\cap\ann(y)$. This shows that $d(x,y)\le2$ if and only if $xy=0$ or $\ann(x)\cap\ann(y)\neq0$.

Next we claim that $\ann(x)\cap\ann(y)\neq0$ if and only if
$\mathcal N(x)\cap\mathcal N(y)\neq\emptyset$. By 2.10,
$\ann(x)\cap\ann(y)=\prod_t\bigl(\ann_{T_t}(x_t)\cap\ann_{T_t}(y_t)\bigr)$, which is non-zero
exactly when some factor is non-zero.   Let $0 \neq  z \in \ann(x)\cap\ann(y).$ So, there exists $t$ such that $z_t \neq 0$.  Since $x_tz_t=0=y_tz_t$, it follows that $x_t$ and $y_t$ cannot be units. So, $t \in  \mathcal N(x)\cap\mathcal N(y).$ 
Conversely, let $t\in\mathcal N(x)\cap\mathcal N(y).$ If $T_t=K_j$ then $x_t=y_t=0$ and the $t$-th factor
is $K_j\neq0$, while if $T_t=R_i$ then $x_t,y_t\in\mathfrak m_i$ and the $t$-th factor
contains $\Soc R_i\neq0$.

Combining the two claims with $\dg\Gamma(R)\le3$ gives the lemma.
\end{proof}

\begin{theorem}\label{thm:core}
Let $R$ be as in Notation \ref{not:prod}. Then
\[
   C(\Gamma(R))=\bigl(N(R)\cup\varepsilon_1R\cup\cdots\cup\varepsilon_mR\bigr)\setminus\{0\}.
\]
Equivalently, a vertex $x$ lies in the core if and only if either every coordinate of $x$ is a
non-unit, i.e., $x$ is nilpotent, or $x=\varepsilon_jc$ for some $j$ and some
$c\in K_j^{\times}$. Consequently
$B(\Gamma(R))=Z^{*}(R)\setminus\bigl(N(R)\cup\varepsilon_1R\cup\cdots\cup\varepsilon_mR\bigr)$.
\end{theorem}

\begin{proof}
Let $x \in Z^{*}(R).$ We make the following cases and show in each case whether or not $x$ belongs to	$C(\Gamma(R)).$

\noindent \underline{Case (i)} $\mathcal U(x)=\emptyset$, i.e.,  $x\in N(R)\setminus\{0\}$ : In this case, $x = (x_1, \ldots, x_k, 0, \ldots, 0)$ where $ x_i \in \m_i.$ 	Note that $\mathcal N(x) = \{1, 2, \ldots, n\}$ and  hence for any $ y \in Z^*(R)$ with $ x \neq y,$ we have $\mathcal N(x) \cap \mathcal N(y) \neq \emptyset.$  So, by Lemma \ref{lem:dist},  $d(x,y)\le 2$ for all $y \in V(\Gamma(R))$ and hence 	$x\in C(\Gamma(R))$.

\smallskip
\noindent \underline{Case (ii)} $\mathcal U(x)\neq\emptyset$ and $x_t\neq 0$ for some 	$t\in\mathcal N(x)$: Since $x$ is a vertex of $\Gamma(R)$ we have $\mathcal N(x)\neq\emptyset$, so
$n\ge 2$. Fix $t_0\in\mathcal U(x)$ and let $y$ be defined by $y_{t_0}=0$ and $y_s=1$ for
$s\neq t_0$. Then $y\neq 0$, $\mathcal N(y)=\{t_0\}\subseteq\mathcal U(x)$, and $y\neq x$
because $y_t=1$ while $x_t$ is a non-unit.  As $y_t=1$ and $x_t \neq 0$ is a non-unit, we have $xy\neq 0. $ Since $\mathcal N(y)=\{t_0\}\subseteq \mathcal U(x),$ we have $N(y) \cap N(x) = \emptyset.$  So, by Lemma \ref{lem:dist}, $d(x,y)=3.$ This shows that $x\in B(\Gamma(R))$.

\smallskip
\noindent \underline{Case (iii)} $\mathcal U(x)\neq\emptyset$ and $x_t=0$ for all	$t\in\mathcal N(x)$ : \\	
\noindent \underline{Subcase (i)}  $|\mathcal U(x)|\ge 2$: Choose $s, t \in \mathcal U(x)$ with $s \neq t$ and let $y_s=0$,
$y_r=1$ for $r\neq s$. Then, $\mathcal N(y)=\{s\}$, $y\neq x$ (as $y_s=0$ while
$x_s$ is a unit), and
$x_ty_t\neq 0.$ As,  $\mathcal N(y) \subset \mathcal U(x)$ we have, $\mathcal N(x) \cap \mathcal N(y) = \emptyset.$  Hence, by Lemma \ref{lem:dist} $d(x,y)=3$ giving us $x\in B(\Gamma(R))$.

\noindent \underline{Subcase (ii)}  $\mathcal U(x)=\{t_0\}$ with $T_{t_0}=R_i$ non-reduced:  Note first that $\mathcal N (x)= \{1, 2, \ldots, n\} \setminus \{ t_0\}.$ Choose
$y_{t_0}\in\mathfrak m_i\setminus\{0\}$ and $y_r=1$ for $r\neq t_0$. Then,
$\mathcal N(y)=\{t_0\}$ and $x_{t_0}y_{t_0}\neq 0$ because $x_{t_0}$ is a unit. So, we have $xy \neq 0.$ Also, $x \neq y$ as $y_{t_0}$ is non-unit while $x_{t_0}$ is unit. It is easy to see that $\mathcal N(x) \cap \mathcal N(y) =\emptyset.$ So, by Lemma \ref{lem:dist} we obtain	$x\in B(\Gamma(R))$ .

\noindent \underline{Subcase (iii)} If $\mathcal U(x)=\{t_0\}$ with $T_{t_0}=K_j$ a field: Since already $x_t=0$ for all	$t\in\mathcal N(x)$ we find that in this subcase  $x=\varepsilon_jc$ with
$c\in K_j^{\times}$. We claim that $x\in C(\Gamma(R))$. Suppose on the contrary that $x\in B(\Gamma(R))$. So,  there exists $y$ in $\Gamma(R)$ such that $d(x, y) =3.$ By Lemma \ref{lem:dist} we obatin $\mathcal N(x)\cap\mathcal N(y)=\emptyset$ and $ xy\neq 0.$  Thus, $\mathcal N(y) \subseteq \mathcal U(x)= \{t_0\}$ giving us that either $\mathcal N(y) = \emptyset $ or $\mathcal N(y) = \{t_0\}.$ Since $y$ is a zero-divisor, $\mathcal N(y) = \emptyset $ is not possible. Hence, $\mathcal N(y) = \{t_0\}.$ Since, $T_{t_0}=K_j$ is a field it follows that $y_{t_0} =0.$
Since $x_t =0$ for all $ t \neq t_0$ this gives $xy = 0,$ contradicting $x y \neq 0.$   This proves that $x\in C(\Gamma(R))$.\\
\indent	Now combining the cases (i), (ii) and (iii) we get the desired description of $C(\Gamma(R))$ and $B(\Gamma(R)).$
\end{proof}

\begin{remark}
For $m=0$ the theorem says $C(\Gamma(R))=N(R)\setminus\{0\}$. Note that for $n=1$ we recover the familiar fact that $\dg\Gamma(R)\le2$ for an Artinian local
ring, as in this case $C(\Gamma(R))=Z^{*}(R)$. For $k=0$, i.e., $R$ is reduced, the core consists of
the vertices with exactly one non-zero coordinate.
\end{remark}

\section{Counting the core-dominating vertices}\label{sec:count}

\begin{definition}\label{def:D}
Let $G$ be a connected graph with $\dg G\le3$. A vertex $v\in V(G)$ is \emph{core-dominating}
if $v$ is adjacent to every vertex of $C(G)\setminus\{v\}$. We write $D(G)$ for the set of
core-dominating vertices. Being defined purely in terms of the core and of adjacency,
$D(\cdot)$ is a graph invariant: if $G\cong H$ then $|D(G)|=|D(H)|$.
\end{definition}

\begin{lemma}\label{lem:DinC}
$D(\Gamma(R))\subseteq C(\Gamma(R))$, and likewise for $\GA(R)$ and $\GE(R)$.
\end{lemma}

\begin{proof}
Let $z\in B(\Gamma(R))$ and pick $t\in\mathcal N(z)$. If $t\le k$ put $y=s_te_t$ with
$0\neq s_t\in\Soc R_t$, and if $t>k$ put $y=\varepsilon_{t-k}$. In both cases $yz=0$ and
$y\in C(\Gamma(R))$ by Theorem \ref{thm:core}, so every boundary vertex has a neighbour in the
core. Hence if $x$ is adjacent to all of $C(\Gamma(R))\setminus\{x\}$, then $d(x,z)\le2$ for
$z\in B(\Gamma(R))$ and $d(x,z)\le1$ for $z\in C(\Gamma(R))$, so $x\in C(\Gamma(R))$. The
arguments for $\GA(R)$ and $\GE(R)$ are identical, using Corollary \ref{cor:core-transfer}.
\end{proof}

\begin{proposition}\label{prop:Dgamma}
Let $R$ be as in Notation \ref{not:prod} and put $t:=\#\{j:|K_j|=2\}$. Then
\[
   D(\Gamma(R))=\bigl(\Soc R_1\times\cdots\times\Soc R_k\times0\times\cdots\times0\bigr)
   \setminus\{0\}\ \cup\ \{\varepsilon_j:\ K_j\cong\mathbb F_2\},
\]
and consequently $\displaystyle|D(\Gamma(R))|=\prod_{i=1}^{k}\bigl|\Soc R_i\bigr|-1+t$.
\end{proposition}

\begin{proof}
	$(\supseteq)$ Let $x\in\prod_i\operatorname{Soc}R_i\times 0$ be non-zero. For
$y\in N(R)$ we have $xy=0$ because $\operatorname{Soc}R_i$ annihilates $\mathfrak m_i$, and
$x\varepsilon_jc=0$ by disjointness of supports. So, by Theorem \ref{thm:core} it follows that $x$ is core-dominating vertex, i.e., $x \in D(\Gamma(R)).$ If $x =\varepsilon_j $ for some $j$ such that  $K_j\cong\mathbb F_2$ then $\varepsilon_jR\setminus\{0\}=\{\varepsilon_j\}$, and
$\varepsilon_j$ annihilates $N(R)$ and every $\varepsilon_{j'}c'$ with $j'\neq j$. Hence, $x \in D(\Gamma(R)).$ 

To prove the other containment	$\subseteq , $ let $x\in D(\Gamma(R)).$ By Lemma \ref{lem:DinC} and Theorem
\ref{thm:core} there are two possibilities: either $x=\varepsilon_jc$ with $c\in K_j^{\times}$ or $x\in N(R)\setminus\{0\}$.

\noindent Case(i): \emph{$x=\varepsilon_jc$ with $c\in K_j^{\times}$.} If $|K_j|>2$, choose $c'\in K_j^{\times}$
with $c'\neq c$. Then $\varepsilon_jc'$ is a core vertex distinct from $x$ and
$x\cdot\varepsilon_jc'=\varepsilon_jcc'\neq 0$, a contradiction. Hence,   $|K_j|=2$ and
$x=\varepsilon_j$.

\noindent Case(ii): $x\in N(R)\setminus\{0\}$. Suppose $x_i\notin\operatorname{Soc}R_i$ for some $i \in \{1, 2, \ldots , k\}$, and
choose $y_i\in\mathfrak m_i$ with $x_iy_i\neq 0$. Put $y=y_ie_i\in N(R)\setminus\{0\}$.
If $y\neq x$ we contradict core-domination. If $y=x$, then $x=x_ie_i$ with $x_i^2\neq 0$;
replace $y$ by $y'=(x_i+s_i)e_i$ where $0\neq s_i\in\operatorname{Soc}R_i$. Then $y'\neq x$;
also $y'\neq 0$, for $x_i=-s_i$ would give $x_i\in\operatorname{Soc}R_i$ and hence
$x_i^2=0.$ Since $x_i+s_i \in \m_i$  it follows from Theorem \ref{thm:core}  that $y'=(x_i+s_i)e_i \in C(\Gamma(R)).$ Since, $xy'=x_i^2+x_is_i=x_i^2\neq 0$, we find that $x$ is not adjacent to $y' \in C(\Gamma(R))$ again a contradiction. Therefore
$x_i\in\operatorname{Soc}R_i$ for every $i$.
\end{proof}

\begin{lemma}\label{lem:dS}
Let $(S,\mathfrak m)$ be a non-reduced Artinian local ring and let $d(S)$ be the number of
universal vertices of $\GE(S)$. Then $d(S)=2$ if $|V(\GE(S))|=2$, and $d(S)=1$ otherwise.
\end{lemma}

\begin{proof}
By Theorem \ref{thm:core} with $n=1$ we have $C(\Gamma(S))=V(\Gamma(S))$, so for local rings
core-domination in $\GE(S)$ is the same as universality. If $|V(\GE(S))|=1$, the unique vertex
is vacuously universal. If $|V(\GE(S))|=2$, then $\GE(S)$ being connected gives $\GE(S)\cong K_2$
and both vertices are universal. If $|V(\GE(S))|>2$, then by Lemma \ref{lem:unique-univ} the
only universal vertex is $[s]_E=\Soc S\setminus\{0\}$.
\end{proof}

\begin{proposition}\label{prop:DE}
Let $R$ be as in Notation \ref{not:prod}. Then
\[
   D(\GE(R))=\Bigl\{\bigl[\textstyle\sum_{i\in S}s_ie_i\bigr]_E:\
   \emptyset\neq S\subseteq\{1,\dots,k\},\ 0\neq s_i\in\Soc R_i\Bigr\}
   \cup\bigl\{[\varepsilon_j]_E:1\le j\le m\bigr\}\cup\Delta ,
\]
where $\Delta=\emptyset$ unless $k=1$ and $|V(\GE(R_1))|=2$, in which case $\Delta$ consists of
the second vertex of $\GE(R_1)$. Consequently
\[
   \bigl|D(\GE(R))\bigr|=2^{\,k}-1+m+\delta,\qquad
   \delta=\begin{cases}1,&k=1\text{ and }|V(\GE(R_1))|=2,\\0,&\text{otherwise.}\end{cases}
\]
In particular $|D(\GE(R))|$ is finite.
\end{proposition}

\begin{proof}
By Theorem \ref{thm:core} and by Corollary \ref{cor:core-transfer}, the core of $\Gamma_E(R)$  consists
of the classes contained in $N(R)\setminus\{0\}$ together with the $m$ classes
$[\varepsilon_j]_E$.

We first show that the classes $[\varepsilon_j]_E$ are always core-dominating. Indeed, the only core
elements not annihilated by $\varepsilon_jc$ are the elements $\varepsilon_jc'$, and these
all lie in $[\varepsilon_jc]_E$. Note that no condition on $|K_j|$ occurs here and this is the
point at which $\Gamma_E(R)$ differs from $\Gamma(R)$
(see Proposition \ref{prop:Dgamma}).

Next we show that the socle classes are core-dominating. If
$0\neq x\in\prod_i\operatorname{Soc}R_i\times 0$, then $x$ annihilates the whole core, so
$[x]_E$ is core-dominating. Any two elements in $\prod_i\operatorname{Soc}R_i\times 0$  are equivalent under $\approx$ if and only if
they have the same support $S\subseteq\{1,\dots ,k\}$, because
$\operatorname{ann}(x)=\prod_{i\in S}\mathfrak m_i\times\prod_{i\notin S}R_i\times\prod_jK_j$;
this gives $2^{k}-1$ distinct classes.

We now show that no further classes occur when $k\ge 2$. Let $x\in N(R)\setminus\{0\}$ with
$x_i\notin\operatorname{Soc}R_i$ for some $i\le k$, choose $y_i\in\mathfrak m_i$ with
$x_iy_i\neq 0$, and fix $l\le k$ with $l\neq i$. If $x_l\neq 0$, put $y=y_ie_i$; then
$\operatorname{ann}(y)_l=R_l\neq\operatorname{ann}(x_l)=\operatorname{ann}(x)_l$, so
$[y]_E\neq[x]_E$, while $xy=x_iy_ie_i\neq 0$. So, $[x]_E$ is a core vertex not	adjacent to the core vertex $[y]_E.$
If $x_l=0$, put
$y=y_ie_i+s_le_l$ with $0\neq s_l\in\operatorname{Soc}R_l$; then
$\operatorname{ann}(y)_l=\mathfrak m_l\neq R_l=\operatorname{ann}(x)_l$, so again
$[y]_E\neq[x]_E$, while $xy=x_iy_ie_i\neq 0$. Thus, in both the cases $[x]_E$ is a core vertex not
adjacent to the core vertex $[y]_E$, so $[x]_E\notin D(\Gamma_E(R))$.

Finally, we deal with the case $k=1$. Here $N(R)=\mathfrak m_1\times 0$ and every $x\in N(R)$ automatically
annihilates all the $\varepsilon_jc$. Moreover,
$\operatorname{ann}_R(x_1e_1)=\operatorname{ann}_{R_1}(x_1)\times\prod_{j}K_j$, so the classes
of $\Gamma_E(R)$ inside $N(R)\setminus\{0\}$ correspond bijectively with
adjacency, to the vertices of $\Gamma_E(R_1)$. Hence their contribution to
$D(\Gamma_E(R))$ is $d(R_1)=1+\delta$ by Lemma \ref{lem:dS}. The case $k=0$ gives only the $m$ classes $[\varepsilon_j]_E$, in agreement with $(2^{0}-1)+m=m$.

\end{proof}

\begin{proposition}\label{prop:DA}
Let $R$ be as in Notation \ref{not:prod} and let $P_i=P(R_i)$ be as in Lemma
\ref{lem:socorbits}. Then $D(\GA(R))$ contains the $m$ classes $[\varepsilon_j]_A$ together
with all associate classes of non-zero elements of $\Soc R_1\times\cdots\times\Soc R_k\times0
\times\cdots\times0$, and the latter are $\prod_{i=1}^{k}(1+P_i)-1$ in number. Hence
\[
   \bigl|D(\GA(R))\bigr|\;\ge\;\prod_{i=1}^{k}\bigl(1+P_i\bigr)-1+m .
\]

\end{proposition}

\begin{proof}
By Theorem \ref{thm:core} and Corollary \ref{cor:core-transfer}, the core of $\GA(R)$ consists
of the associate classes contained in $N(R)\setminus\{0\}$ together with the classes
$[\varepsilon_jc]_A= \varepsilon_jK_j^\times=[\varepsilon_j]_A$ for $j=1, \ldots, m$. It follows from Corollary \ref{cor:core-transfer} that  $[\varepsilon_j]_A \in C(\Gamma_A(R))$  for $j=1, \ldots, m$. We notice that $[\varepsilon_j]_A$ is adjacent to $[\varepsilon_s]_A$ for $ s \neq j$ and $[\varepsilon_j]_A$ is adjacent to $[x]_A$ for all $x \in N(R).$ So, by Theorem \ref{thm:core} and Corollary \ref{cor:core-transfer} it follows that $[\varepsilon_j]_A \in D(\Gamma_A(R))$ for $j=1, \ldots, m$. 

 Next, we note that if $0\neq x\in\prod_i\Soc R_i
\times0$, then $x$ annihilates the whole core, so $[x]_A\in D(\GA(R))$.
We count these socle classes. As $R^{\times}=\prod_iR_i^{\times}\times\prod_jK_j^{\times}$
acts coordinatewise and the coordinates of $x$ outside its support are $0$, we have
$[x]_A=\prod_{i\in S}[s_i]_A$ for $x=\sum_{i\in S}s_ie_i$ with $0\neq s_i\in\Soc R_i$.  Note that there are $1+P_i$ choices in each coordinate where $P_i$ non-zero choices  come from  $[s_i]_A$  ($0\neq s_i\in\Soc R_i$) by Lemma \ref{lem:socorbits} and one choice comes from $[0]_A$, the associate class corresponding to $0$ in $R_i.$ Since, $[0]_A \notin \Gamma_A(R)$ we see that there are $\prod_{i=1}^{k}(1+P_i)-1$ elements coming from $\prod_i\Soc R_i
\times0$.
Hence,  $ |D(\GA(R))| \geq \prod_{i=1}^{k}\bigl(1+P_i\bigr)-1+m.$

\end{proof}
Recall that in Theorem \ref{thm:gor-nat} we proved that the natural map $\varphi:V(\GA(R))\to V(\GE(R))$,  is a
graph isomorphism if and only if $R$ is Gorenstein. The question of whether an arbitrary isomorphism $\GA(R)\cong\GE(R)$ forces  the natural map $\varphi$ to be an isomorphism and  the ring $R$ to be Gorenstein is answered in the next theorem. This is first of our two main results.
\begin{theorem}\label{thm:mainA}
Let $R$ be an Artinian ring. Then the following are equivalent.
\begin{enumerate}
\item[(1)] $R$ is Gorenstein;
\item[(2)] $\bigl|D(\GA(R))\bigr|=\bigl|D(\GE(R))\bigr|$;
\item[(3)] $\GA(R)\cong\GE(R)$ as graphs.
\item[(4)] the natural map $\varphi:V(\GA(R))\to V(\GE(R))$, $\varphi([a]_A)=[a]_E$, is a
graph isomorphism.
\end{enumerate}
\end{theorem}

\begin{proof}
If $R$ is a field then all four statements hold trivially, the two graphs being empty. So, we
may assume $R$ is not a field and use Notation \ref{not:prod}.

(1)$\Rightarrow$(3). If $R$ is Gorenstein then $\varphi$ is a graph isomorphism by Theorem
\ref{thm:gor-nat}.

(3)$\Rightarrow$(2). Immediate, since $D(\cdot)$ is a graph invariant by Definition
\ref{def:D}.

(2)$\Rightarrow$(1). Assume $R$ is not Gorenstein. As every field is Gorenstein and a product
is Gorenstein exactly when each factor is, some factor $R_j$ with $j\le k$ is non-Gorenstein. Since $R_j$ is non-Gorenstein non-reduced local ring, by Lemma \ref{lem:socorbits} we get $P_j =P(R_j) \geq 3$. We show
$|D(\GA(R))|>|D(\GE(R))|$, which contradicts (2). By Propositions \ref{prop:DE} and
\ref{prop:DA} it suffices to show
\[
   \prod_{i=1}^{k}\bigl(1+P_i\bigr)-1+m\;>\;2^{\,k}-1+m+\delta .
\]
Suppose first $k\ge2$. Then $\delta=0$, and using $P_i\ge1$ for $i\neq j$ together with
$P_j\ge3$ we get $\prod_i(1+P_i)\ge(1+P_j)\cdot2^{\,k-1}\ge4\cdot2^{\,k-1}=2^{\,k+1}$, so the
left-hand side is at least $2^{\,k+1}-1+m>2^{\,k}-1+m$, as required. Now suppose $k=1$; then
$j=1$ and $\delta\le1$, so the left-hand side equals $P_1+m\ge3+m>2+m\ge2^{1}-1+m+\delta$.

(1)$\iff$(4) follows from Theorem \ref{thm:gor-nat}. This proves the theorem.
\end{proof}

\section{The question of Anderson and LaGrange}\label{sec:lagrange}

Anderson and LaGrange in \cite[Question 2.12]{ALB} and \cite[Question
4.10]{LaG16}, ask which rings satisfy $\Gamma(R)\cong\GE(R)$. They show that if $\Gamma_E(R)$ is finite then, $\Gamma(R)\cong\GE(R)$ if and only if $R$ is Boolean ring or  $R \in \{ \mathbb Z_4, ~ \Z_2[x]/(x^2) \}.$  They also show that for reduced rings, $\Gamma(R)\cong\GE(R)$ if and only if $R$ is a Boolean ring. In this section we answer this question 
for Artinian rings.  The following lemma easily follows from \cite[Theorem 2.9]{ALB} but we prove it here for the sake of completeness.
\begin{lemma}\label{lem:eta-iso}
Let $R$ be a commutative ring with $Z^{*}(R)\neq\emptyset$. Then $\eta:\Gamma(R)\to\GE(R)$ is
a graph isomorphism if and only if $|[a]_E|=1$ for every $a\in Z^{*}(R)$.
\end{lemma}

\begin{proof}
If every class is a singleton then $\eta$ is bijective, and for $a\neq b$ we have $ab=0$ if
and only if $[a]_E\neq[b]_E$ are adjacent; so $\eta$ is an isomorphism. Conversely, suppose $\eta$
is an isomorphism. If $a,b\in Z^{*}(R)$ are elements of $[a]_E$ then we have  $\ann(a)=\ann(b).$  Note that $\eta(a)=\eta(b)$ and as $\eta$ is isomorphism we have $a=b$. This shows that $|[a]_E|=1.$
\end{proof}

\begin{theorem}\label{thm:mainB}
Let $R$ be an Artinian ring which is not a field. Then the following are equivalent.
\begin{enumerate}
\item[(1)] $\Gamma(R)\cong\GE(R)$ as graphs;
\item[(2)] $\eta:\Gamma(R)\to\GE(R)$ is a graph isomorphism;
\item[(3)] every annihilator class of $R$ is a singleton;
\item[(4)] $R\cong\Z_2^{\,n}$ for some $n\ge2$, or $R\cong\Z_4$, or $R\cong\Z_2[x]/(x^2)$.
\end{enumerate}
\end{theorem}

\begin{proof}
 It is easy to verify (4)$\Rightarrow$(3). Equivalence  $(3) \Leftrightarrow$(2) follows from Lemma \ref{lem:eta-iso}. Implication 
  $(2) \Rightarrow$(1) is obvious. It remains to verify 
(1)$\Rightarrow$(4). We assume $\Gamma(R)\cong\GE(R)$ as graphs. We use Notation \ref{not:prod} and  recall that  $R\;\cong\;R_1\times\cdots\times R_k\times K_1\times\cdots\times K_m$ where $R_i$'s are Artinian local rings and $K_j$'s are fields.

\underline{Claim 1} $R_i$ is finite for $1 \leq i \leq k  $: Since $D(\cdot)$ is a graph invariant,
$|D(\Gamma(R))|=|D(\GE(R))|$, and the right-hand side is finite by Proposition \ref{prop:DE}.
By Proposition \ref{prop:Dgamma} we have $|D(\Gamma(R))|=\prod_{i=1}^k|\Soc R_i|-1+t$, so every
$\Soc R_i$ is finite. As $\Soc R_i$ is a non-zero $k_i$-vector space for each $i$ with  $1 \leq i \leq k  $ we find that   $k_i$ is finite field. Since each $R_i$ is an Artinian local ring with finite residue field $k_i$ we obtain $R_i$ is finite.

\underline{Claim 2} $R$ is finite: Consider the collection $\{\ann(x) \mid x \in Z^*(R)\}$ of annihilator ideals of $R.$ 

 By 2.10 any annihilator ideal $\ann(x)$ of $R$ is the product 
$\ann(x)=\prod_t\ann_{T_t}(x_t)$. For each $1 \leq i \leq k,$  $R_i$ is finite by claim 1 and so has finitely many ideals, while each field $K_j$ has exactly two annihilator ideals, namely $K_j$ and $0$. Hence, $R$ has
only finitely many annihilator ideals and $\GE(R)$ is a finite graph. Since $\Gamma(R)\cong\GE(R)$ we see that 
$\Gamma(R)$ is finite, i.e., $Z^{*}(R)$ is finite. If some $K_j$ were infinite, then $n\ge2$ because $R$
is not a field, and the infinitely many elements $\varepsilon_jc$ with
$c\in K_j\setminus\{0\}$ would all be non-zero zero-divisors, a contradiction. So every $K_j$
is finite for each $j=1, \ldots, m$ and $R$ is finite.

 By  \cite[Corollary 2.10]{ALB} we see that  $R$ is a Boolean ring or $R\cong\Z_4$, or $R\cong\Z_2[x]/(x^2)$. Note that $R$ is not a field by assumption. By claim 2, $R$ is a finite ring that is not a field. So, if $R$ is a Boolean ring then $R\cong\Z_2^{\,n}$ for some $n\ge2$. This proves (4).
\end{proof}


\subsection*{Statement on the use of AI} Generative AI tools were used for editorial assistance, drafting refinement, and grammar checking during the preparation of this manuscript. All underlying ideas, mathematical proofs, and conceptual work were developed solely by the authors.

\subsection*{Acknowledgement}
The first author would like to thank the Science and Engineering Research Board (SERB,
Government of India) for financial assistance under the scheme CRG/2022/002184.


\begin{thebibliography}{99}

\bibitem{ALB} D. F. Anderson and J. D. LaGrange, \emph{Commutative Boolean monoids, reduced
rings, and the compressed zero-divisor graph}, J. Pure Appl. Algebra \textbf{216} (2012) (7)
1626--1636.

\bibitem{LaG16} D. F. Anderson and J. D. LaGrange, \emph{Some remarks on the compressed zero-divisor graph},
J. Algebra \textbf{447} (2016) 297--321.

\bibitem{ALS} D. F. Anderson, R. Levy and J. Shapiro, \emph{Zero-divisor graphs, von Neumann
regular rings, and Boolean algebras}, J. Pure Appl. Algebra \textbf{180} (2003) 221--241.

\bibitem{AL} D. F. Anderson and P. S. Livingston, \emph{The zero-divisor graph of a
commutative ring}, J. Algebra \textbf{217} (1999) 434--447.

\bibitem{Ax} M. Axtell, N. Baeth and J. Stickles, \emph{Cut structures in zero-divisor graphs
of commutative rings}, J. Commut. Algebra \textbf{8}(2) (2016) 143--171.

\bibitem{Bass} H. Bass, \emph{On the ubiquity of Gorenstein rings}, Math. Z. \textbf{82}
(1963) 8--28.

\bibitem{Beck} I. Beck, \emph{Coloring of commutative rings}, J. Algebra \textbf{116} (1988)
208--226.

\bibitem{BH} W. Bruns and J. Herzog, \emph{Cohen--Macaulay Rings}, Cambridge Studies in
Advanced Mathematics, vol. 39, Cambridge University Press, Cambridge, 1993.

\bibitem{CSSS} J. Coykendall, S. Sather-Wagstaff, L. Sheppardson and S. Spiroff, \emph{On zero
divisor graphs}, Progress in Commutative Algebra 2, 241--299, Walter de Gruyter, Berlin, 2012.

\bibitem{DJS} A. Duric, S. Jevdenic and N. Stopar, \emph{Categorial properties of compressed
zero-divisor graphs of finite commutative rings}, J. Algebra Appl. \textbf{20}(5) (2021)
2150069, 16 pp.

\bibitem{Kadu} G. S. Kadu, \emph{On the problem of zero-divisors in Artinian rings}, Comm.
Algebra \textbf{49}(11) (2021) 4957--4969.

\bibitem{LaG07} J. D. LaGrange, \emph{Complemented zero-divisor graphs and Boolean rings},
J. Algebra \textbf{315} (2007) 600--611.

\bibitem{LaG08} J. D. LaGrange, \emph{The cardinality of an annihilator class in a von Neumann
regular ring}, Int. Electron. J. Algebra \textbf{4} (2008) 63--82.



\bibitem{Mulay} S. B. Mulay, \emph{Cycles and symmetries of zero-divisors}, Comm. Algebra
\textbf{30}(7) (2002) 3533--3558.


\bibitem{Red} S. P. Redmond,  \emph{Central Sets and Radii of the Zero-Divisor Graphs of Commutative Rings}, Comm. Algebra, \textbf{34}(7), (2006) 2389-2401.




\bibitem{SW} S. Spiroff and C. Wickham, \emph{A zero divisor graph determined by equivalence
classes of zero divisors}, Comm. Algebra \textbf{39}(7) (2011) 2338--2348.

\bibitem{TLW} G. Tang, G. Lin and Y. Wu, \emph{Associate class graph of zero-divisors of a
commutative ring}, J. Algebra Appl. \textbf{19}(08) (2020) 2050155.

\end{thebibliography}
\end{document}